\documentclass[UTF-8,reqno]{amsart}
\usepackage{enumerate, bbm}
\usepackage{amssymb,url,color, booktabs}
\usepackage{mathrsfs}
\usepackage{soul,cancel}
\usepackage{color}
\usepackage[colorlinks=true]{hyperref}
\hypersetup{
    linkcolor=blue,          
    citecolor=red,        
    filecolor=blue,      
    urlcolor=cyan
}
\definecolor{MyDarkBlue}{cmyk}{0.8,0.3,0.8,0.4}
\definecolor{yellow}{rgb}{0.99,0.99,0.70}
\definecolor{white}{rgb}{1.0,1.0,1.0}
\definecolor{black}{rgb}{0.00,0.00,0.00}
\definecolor{backgroundcolor}{RGB}{199,238,206}

\numberwithin{equation}{section}

\newcommand{\be}{\begin{eqnarray}}
\newcommand{\ee}{\end{eqnarray}}
\newcommand{\ce}{\begin{eqnarray*}}
\newcommand{\de}{\end{eqnarray*}}
\newtheorem{theorem}{Theorem}[section]
\newtheorem{lemma}[theorem]{Lemma}
\newtheorem{remark}[theorem]{Remark}
\newtheorem{definition}[theorem]{Definition}
\newtheorem{proposition}[theorem]{Proposition}
\newtheorem{Examples}[theorem]{Example}
\newtheorem{corollary}[theorem]{Corollary}

\def\nor{|\mspace{-3mu}|\mspace{-3mu}|}

\def\eps{\varepsilon}
\def\e{\mathrm{e}}
\def\supp{\mathrm{supp}}

\def\dif{{\mathord{{\rm d}}}}

\def\max{{\mathord{{\rm max}}}}

\def\bba{{\boldsymbol{a}}}

\def\bbk{{\boldsymbol{k}}}

\def\bbp{{\boldsymbol{p}}}
\def\bbr{{\boldsymbol{r}}}

\def\bbone{{\boldsymbol{1}}}
\def\bb2{{\boldsymbol{2}}}
\def\no{\nonumber}
\def\={&\!\!=\!\!&}

\def\bx{{\mathbf{x}}}

\def\bB{{\mathbf B}}

\def\b1{{\mathbbm 1}}

\def\cA{{\mathcal A}}

\def\cD{{\mathcal D}}

\def\cG{{\mathcal G}}

\def\cI{{\mathcal I}}

\def\cR{{\mathcal R}}

\def\mE{{\mathbb E}}

\def\mL{{\mathbb L}}

\def\mN{{\mathbb N}}

\def\mR{{\mathbb R}}

\def\sI{{\mathscr I}}

\def\sS{{\mathscr S}}

\def\sV{{\mathscr V}}

\def\geq{\geqslant}
\def\leq{\leqslant}
\def\<{{\langle}}
\def\>{{\rangle}}
\def\({{\big(}}
\def\){{\big)}}
\def\[{{\Big[}}
\def\]{{\Big]}}

\def\b{\beta}
\def\de{\delta}
\def\g{\gamma}

\def\k{\kappa}
\def\l{\lambda}

\def\s{\sigma}

\def\ff{\frac}
\def\nn{\nabla}

\def\p{\partial}
\def\div{\mathord{{\rm div}}}

\def\bt{\begin{theorem}}
\def\et{\end{theorem}}
\def\bl{\begin{lemma}}
\def\el{\end{lemma}}
\def\br{\begin{remark}}
\def\er{\end{remark}}
\def\bpf{\begin{proof}}
\def\epf{\end{proof}}
\def\bx{\begin{Examples}}
\def\ex{\end{Examples}}
\def\bd{\begin{definition}}
\def\ed{\end{definition}}
\def\bp{\begin{proposition}}
\def\ep{\end{proposition}}
\def\bc{\begin{corollary}}
\def\ec{\end{corollary}}

\def\wt{\widetilde}

\allowdisplaybreaks

\begin{document}

\title{Kinetic Fokker--Planck Equations with Drifts in a Supercritical Range}
\author[Z. Chen]{Zikai Chen}
\author[C. Ren]{Chongyang Ren}

\thanks{Zikai Chen: School of Mathematics and Statistics, Wuhan University, Wuhan, Hubei, 430072, China, and Department of Mathematics Sciences, Graduate School of Science, Kyoto University, Kyoto, 606-8502, Japan. E-mail: chenzikai@whu.edu.cn}

\thanks{Chongyang Ren: School of Mathematical Sciences, University of Science and Technology of China, Hefei, Anhui 230026, P. R. China, Email: rcy.math@ustc.edu.cn}

\thanks{The first author acknowledges the support from China Scholarship Council (No. 202506270058). The second author was supported by China Postdoctoral Science Foundation (No. 2026M793409).}

\begin{abstract}
We investigate kinetic Fokker--Planck equations with rough divergence-free drifts in anisotropic mixed Lebesgue spaces. The divergence-free structure yields an energy cancellation that remains effective under kinetic localization. Combining this cancellation with anisotropic regularization estimates and a De Giorgi iteration, we establish local boundedness for weak subsolutions and, consequently, global well-posedness for the Cauchy problem in a scaling-supercritical regime.

\end{abstract}


\maketitle
\section{Introduction}
In this paper, we consider the following kinetic Fokker--Planck equation in $\mR^{2d}$:
\begin{align}\label{eqkinetic}
\p_tu=\Delta_vu+v\cdot\nn_xu+b\cdot\nn_vu+f.
\end{align}
Here $b=b(t,x,v)$ is a rough drift that is divergence-free in the velocity variable, and $f=f(t,x,v)$ is an inhomogeneous term. Our main objective is to establish local boundedness and global well-posedness of the equation \eqref{eqkinetic} in a regime in which the drift is supercritical.

 The coupling between velocity diffusion and spatial transport gives the equation its characteristic anisotropic structure and scaling. For $\lambda>0$, define 
$$u_\lambda(t,x,v):=u(\lambda^2t,\lambda^3x,\lambda v),\quad b_\lambda(t,x,v):=\lambda b(\lambda^2t,\lambda^3x,\lambda v),\quad f_\lambda(t,x,v):=\lambda^2 f(\lambda^2t,\lambda^3x,\lambda v).$$
Then $u_\lambda$ solves the same type of equation with $b$ and $f$ replaced by $b_\lambda$ and $f_\lambda$, namely
$$\p_t u_\lambda=\Delta_vu_\lambda+v\cdot\nn_xu_\lambda+b_\lambda\cdot\nn_vu_\lambda+f_\lambda.$$
For $\bba=(3,1)$, $\bbp_1=(p_{1,x},p_{1,v})\in(1,\infty)^2$ and $q_1\in(1,\infty)$, set 
\begin{align*}
S_1:=\bba\cdot\frac d{\bbp_1}:=\frac{3d}{p_{1,x}}+\frac d{p_{1,v}},\ \ \ \
T_1:=\frac2{q_1}. 
\end{align*}
By a change of variables,
$$\|b_\lambda\|_{\mL_t^{q_1}\mL_z^{\bbp_1}}=\lambda^{1-S_1-T_1}\|b\|_{\mL_t^{q_1}\mL_z^{\bbp_1}}.$$
See \eqref{norm} below for the definition of $L^\bbp$-space. Letting $\l\rightarrow0$, we shall categorize the following three cases:
$$\textbf{Subcritical: }S_1+T_1<1;\quad\textbf{Critical: }S_1+T_1=1;\quad\textbf{Supercritical: }S_1+T_1>1.$$
In the supercritical case $S_1+T_1>1$, the norm of $b_\lambda$ blows up as $\lambda\downarrow0$, showing that the drift becomes singular at smaller scales and rescaling cannot reduce the problem to a perturbative regime with a small drift. In particular, the standard strategy based on scaling and perturbation breaks down in the supercritical range. Establishing a priori estimates and global well-posedness in this case is therefore rather difficult.

\subsection{Background and motivation}

Kinetic Fokker--Planck--Kolmogorov equations constitute a basic class of degenerate parabolic equations arising in kinetic theory and stochastic dynamics. Their prototype goes back to Kolmogorov \cite{Kol34}, while the underlying hypoelliptic mechanism is covered by the H\"ormander theory \cite{Hor67}. Although diffusion acts only in the velocity variable, the commutator relation $[\p_{v_i},v\cdot\nn_x]=\p_{x_i}$ transfers regularity from $v$ to $x$. Consequently, the natural geometry and regularity scales are anisotropic, with the position and velocity variables carrying different weights; see, for instance, \cite{Bouchut02} for fundamental hypoelliptic regularity estimates in this setting.

For weak solutions with merely measurable or rough coefficients, the classical uniformly parabolic theory cannot be applied directly. Pascucci and Polidoro \cite{PP04} adapted the Moser iteration to obtain local boundedness for a class of ultraparabolic equations, while Wang and Zhang \cite{WZ11} established H\"older regularity of weak solutions. Moser estimates for degenerate Kolmogorov operators with rough lower-order coefficients satisfying nonnegative divergence conditions were obtained in \cite{APR19}. A kinetic De Giorgi theory was subsequently developed by Golse, Imbert, Mouhot and Vasseur \cite{GIMV19}, leading in particular to Harnack inequalities for kinetic Fokker--Planck equations with rough coefficients. Further quantitative and weak Harnack estimates were obtained in \cite{GM22,GI23}. Related regularity and well-posedness problems for degenerate stochastic equations with low-regularity drifts have also been studied in \cite{CdR17,ZZ24,RZ25}. These works provide the main analytic framework for treating kinetic equations beyond the classical smooth-coefficient setting.

Divergence-free drifts provide a basic example of a mechanism not captured by scaling alone. For nondegenerate drift--diffusion equations, the resulting cancellation in the local energy inequality permits local boundedness and related regularity estimates for rough drifts in ranges extending beyond the usual scaling-critical threshold; see, for example, \cite{AD23}. Related structural and energy-based approaches have recently yielded weak well-posedness results for nondegenerate SDEs with critical and supercritical distributional drifts \cite{GP24,HZ25}.

For kinetic equations with singular coefficients in mixed Lebesgue spaces, Zhang \cite{Zhang25} established local boundedness under singular drifts and distribution-valued inhomogeneous terms. The proof combines regularization estimates for the Kolmogorov semigroup in anisotropic Besov spaces with a De Giorgi iteration. These results, however, are confined to scaling-subcritical regimes. This raises the question of whether additional structural properties of the drift can restore a priori control in the scaling-supercritical regime. The purpose of the present paper is to give a positive answer under the velocity-divergence-free condition. More precisely, we show that the cancellation generated by ${\div}_v b=0$, when combined with anisotropic kinetic smoothing and De Giorgi iteration, yields local boundedness and global weak well-posedness in a nonempty scaling-supercritical range.

\subsection{Main result}

For $i=0,1$, let $q_i\in(1,\infty)$ and $\bbp_i=(p_{i,x},p_{i,v})\in(1,\infty)^2$. We further assume that $q_1>2$ and $\bbp_1\in(2,\infty)^2$. The indices $i=0$ and $i=1$ correspond to the inhomogeneous term $f$ and the drift $b$, respectively. Set 
\begin{align*}
S_i:=\bba\cdot\frac d{\bbp_i}=\frac{3d}{p_{i,x}}+\frac d{p_{i,v}},
\ \ \ \
T_i:=\frac2{q_i}.
\end{align*} 
The basic assumptions are as follows:

\noindent{\bf (${\text{H}}$)}
Assume that $b:\mR_+\times\mR^{2d}\to\mR^d$ and $f:\mR_+\times\mR^{2d}\to\mR$ satisfy $\div_vb=0$ in the distributional sense and

$$b\in L_{\rm loc}^{q_1}\big(\mR_+;L^{\bbp_1}(\mR^{2d})\big),\quad f\in L_{\rm loc}^{q_0}\big(\mR_+;L^{\bbp_0}(\mR^{2d})\big),$$
and that there exists $\b\in(0,1)$ such that
\begin{align}\label{intro-parameter-condition}
\max\Big\{\frac{S_1}{2-T_1},\frac{S_0}{2-T_0}\Big\}<\b<2-S_1-T_1.
\end{align}

\begin{remark}\label{rmk:nonempty-supercritical}
The condition \eqref{intro-parameter-condition} allows for scaling-supercritical drifts. Indeed,
$\frac{(2-T_1)^2}{3-T_1}-(1-T_1)=\frac{1}{3-T_1}>0,$
so one may choose $S_1$ such that
\begin{align*}
1-T_1<S_1<\frac{(2-T_1)^2}{3-T_1}.
\end{align*}
Then $S_1+T_1>1$ and
$\frac{S_1}{2-T_1}<2-S_1-T_1.$
Thus, if in addition $\frac{S_0}{2-T_0}<2-S_1-T_1$, there exists $\beta\in(0,1)$ satisfying \eqref{intro-parameter-condition}.
\end{remark}
\begin{remark}
The velocity-divergence-free condition changes the structure of the localized energy inequality. For a general drift, the drift contribution couples $b$, $u$ and $\nabla_vu$, and its control leads to the usual scaling-subcritical restriction. When $\div_v b=0$, the derivative can instead be transferred entirely to the cutoff, so that the drift enters the localized energy estimate only through a lower-order term. This cancellation is compatible with truncation and kinetic localization and allows the energy and interpolation estimates to close beyond the scaling-critical threshold.
\end{remark}

We can now state a simplified version of the main theorem; the precise uniformly local energy estimate is given in Theorem \ref{thm:df-global}.

\begin{theorem}\label{thm:intro-global}
Under Assumption {\bf (${\text{H}}$)}, the Cauchy problem
$$\p_tu=\Delta_vu+v\cdot\nn_xu+b\cdot\nn_vu+f,\quad u|_{t\leq0}=0,$$
admits a unique global weak solution in the sense of Definition \ref{weaksubsol}, within the class of solutions which are bounded and have finite uniformly local kinetic energy on every finite time interval. Moreover, for every $T>0$ and $t\in[0,T]$,
$$\|u\|_{\mL^\infty((0,t)\times\mR^{2d})}\lesssim_{C_T}\|f\|_{\mL_t^{q_0}\mL_z^{\bbp_0}(0,t)},$$
where $C_T$ depends only on $T,d,q_1,\bbp_1,q_0,\bbp_0,\b$ and $\|b\|_{L^{q_1}((0,T+4);L^{\bbp_1}(\mR^{2d}))}$.
\end{theorem}

\br
The estimates obtained above also provide the basic PDE input for the associated kinetic SDE
$$\dif X_t=V_t\,\dif t,\quad\dif V_t=b(t,X_t,V_t)\,\dif t+\sqrt2\,\dif W_t.$$
Indeed, applying these estimates to smooth approximations of $b$ yields uniform Krylov-type bounds, while the usual weak convergence argument gives the existence of a weak solution to the this SDE under the present assumptions. Since this procedure is standard, we omit the details.
\er

\subsection{Organization of the paper}

The remainder of the paper is organized as follows. In Section~\ref{sec:preliminary}, we introduce the mixed-norm anisotropic Besov spaces and recall the abstract De Giorgi class used in the subsequent analysis. In Section~\ref{sec:local-boundedness}, we define weak sub-solutions and combine regularization estimates for the Kolmogorov semigroup in anisotropic Besov spaces with the divergence-free energy cancellation and a De Giorgi iteration to establish local boundedness and energy estimates. In Section~\ref{sec:global-wellposedness}, we combine the local estimate with Galilean localization and a uniformly local energy argument to prove the global boundedness, existence and uniqueness of weak solutions.

\section{Preliminaries}\label{sec:preliminary}

In this section, we introduce some definitions and preliminary results for later use. We first recall the anisotropic Besov spaces adapted to the kinetic scaling and then state an abstract De Giorgi iteration theorem.

\subsection{Anisotropic Besov spaces}

For $\bbp=(p_x,p_v)\in[1,\infty]^2$, we define the mixed Lebesgue space by (cf. \cite{BP61})
$$L^\bbp(\mR^{2d}):=L^{p_v}\big(\mR_v^d;L^{p_x}(\mR_x^d)\big),$$
equipped with the norm
\begin{align}\label{norm}
\|f\|_{L^\bbp}=\|f\|_\bbp:=\|\|f(\cdot,v)\|_{L_x^{p_x}}\|_{L_v^{p_v}}.
\end{align}
When $\bbp=(p,p)$, we simply write $L^\bbp=L^p$ and $\|f\|_\bbp=\|f\|_p$. 
For simplicity of notation, we define 
\begin{align}
\frac{1}{\bbp}:=\big(\frac{1}{p_x},\frac{1}{p_v}\big),\quad \bba\cdot\frac{d}{\bbp}:=\frac{3d}{p_x}+\frac{d}{p_v},\no
\end{align}
where $\bba:=(3,1)$, and for any $\bbp,\bbp'\in[1,\infty]^2$, 
$$\bbp\leq\bbp'\Leftrightarrow p_x\leq p'_x,\quad p_v\leq p'_v.$$
We also introduce the anisotropic distance:
\begin{align}\label{distance}
|z|_\bba:=|x|^{\ff{1}{3}}+|v|, \ \ z=(x,v)\in\mR^{2d}.
\end{align}

Let $\sS(\mR^{2d})$ be the Schwartz space of all rapidly decreasing functions on $\mR^{2d}$, and let $\sS'(\mR^{2d})$ denote the dual space of $\sS(\mR^{2d})$, known as Schwartz generalized function (or tempered distribution) space. For any $f\in\sS(\mR^{2d})$, we define the Fourier transform $\hat f$ and inverse Fourier transform $\check f$ respectively by
$$\hat f(\xi):=\int_{\mR^{2d}}\e^{-i\xi\cdot z}f(z)\dif z, \quad \xi\in\mR^{2d},$$
$$\check f(z):=\ff{1}{(2\pi)^{2d}}\int_{\mR^{2d}}\e^{i\xi\cdot z}f(\xi)\dif\xi, \quad z\in\mR^{2d}.$$
Additionally, we denote the Banach space of all bounded continuous functions on $\mR^{2d}$ by $C_b=C_b(\mR^{2d})$, and $C_b^\infty=C_b^\infty(\mR^{2d})$ is the space of all functions with bounded derivatives of all orders.

To introduce the anisotropic Besov space, we first give an anisotropic version of dyadic partition of unity. For $r>0$ and $z\in\mR^{2d}$, the ball centered at $z$ and with radius $r$ in terms of the anisotropic distance (see \eqref{distance}) is defined as follows:
$$B_r^\bba(z):=\{z'\in\mR^{2d}:|z-z'|_\bba\leq r\},\quad B_r^\bba:=B_r^\bba(0).$$
Let $\chi_0^\bba$ be a symmetric $C^\infty$-function on $\mR^{2d}$ such that
$$\chi_0^\bba(\xi)=1 \text{ for } \xi\in B_1^\bba \text{\quad and\quad }\chi_0^\bba(\xi)=0 \text{ for } \xi\notin B_{4/3}^\bba.$$
For $j\in\mN$, define 
\begin{align}
\phi_j^\bba(\xi):=\left\{
\begin{aligned}
&\chi_0^\bba(2^{-j\bba}\xi)-\chi_0^\bba(2^{-(j-1)\bba}\xi),\quad &j\geq1,\\
&\chi_0^\bba(\xi),&j=0,
\end{aligned}
\right.\nonumber
\end{align}
where for $s\in\mR$ and $\xi=(\xi_1,\xi_2)$, $2^{s\bba}\xi:=(2^{3s}\xi_1,2^s\xi_2)$. This definition ensures that 
\begin{align}
\sum_{j\geq0}\phi_j^\bba(\xi)=1,\quad \forall\xi\in\mR^{2d},\no
\end{align}
and
\begin{align}
\supp(\phi_j^\bba)\subset\{\xi:2^{j-1}\leq|\xi|_\bba\leq2^{j+1}\},\quad j\geq1;\quad \supp(\phi_0^\bba)\subset B_{4/3}^\bba.\no
\end{align}
For given $j\geq0$, we define the dyadic anisotropic block operator $\cR_j^\bba$ on $\sS'$ as
\begin{align}
\cR_j^\bba f(z):=(\phi_j^\bba\hat f)\check\,(z)=\check\phi_j^\bba*f(z),\no
\end{align}
where the convolution is taken in the distributional sense.

The following Bernstein inequality is a standard result and can be found in \cite{ZZ24}.

\bl\label{bernstein}
For any $\bbk=(k_1,k_2)\in\mN^2, \bbp\leq\bbp'\in[1,\infty]^2$, there is a constant $C=C(\bbk, \bbp, \bbp', d)>0$ such that for all $j\geq0$, 
\begin{align}
\|\nn_x^{k_1}\nn_v^{k_2}\cR_j^\bba f\|_{\bbp'}\lesssim_C 2^{j\bba\cdot(\bbk+\ff{d}{\bbp}-\ff{d}{\bbp'})}\|\cR_j^\bba f\|_\bbp.\no
\end{align}
\el

Now we define the anisotropic Besov spaces as follows (see \cite[Chapter 5]{Triebel06}).

\bd\label{besov}
Let $s\in\mR$, $q\in[1,\infty]$ and $\bbp\in[1,\infty]^2$. The anisotropic Besov space is defined by
$$\bB_{\bbp,\bba}^{s,q}:=\Big\{f\in\sS':\|f\|_{\bB_{\bbp,\bba}^{s,q}}:=\Big(\sum_{j\geq0}(2^{js}\|\cR_j^\bba f\|_\bbp)^q\Big)^{1/q}<\infty\Big\}.$$
For simplicity, we denote $\bB_{\bbp,\bba}^{s}:=\bB_{\bbp,\bba}^{s,\infty}$.
\ed

\subsection{The De Giorgi class}
In this subsection, we recall an abstract De Giorgi’s class introduced in \cite{Zhang25}.
Let $\sI\subset(1,\infty)^N$ be an open multi-index set and $Q:=(Q_\tau)_{\tau\in[1,2]}$ be a family of increasing bounded open set in $\mR^N$ such that
$$Q_\tau\cap Q_\s^c=\varnothing\text{ for }\tau<\s,\quad \cap_{\s>\tau}Q_\s=\bar Q_\tau.$$
\begin{definition}
A function $u\in L^1(Q_2)$ belongs to the De Giorgi class $\cD\cG_{\sI}^+(Q)$ if there are $\bbp_i\in\sI$, $i=1,\ldots,m$, $1\leq j<m$, $\l,\, \cA\geq0$, such that for any $\bbp\in\sI$, there exists a constant $C_\bbp>0$ such that, for every $\bbp\in\sI$, $1\leq\tau<\s\leq2$, and $\k>0$,
$$(\s-\tau)^\l\|\bbone_{Q_\tau}(u-\k)^+\|_{\mL^\bbp}\lesssim_{C_\bbp}\sum_{i=1}^j\|\bbone_{Q_\s}(u-\k)^+\|_{\mL^{\bbp_i}}+\cA\sum_{i=j+1}^m\|\bbone_{\{u>\k\}\cap Q_\s}\|_{\mL^{\bbp_i}}.$$
\end{definition}

We shall use the following De Giorgi iteration, established in \cite[Theorem 2.10]{Zhang25}.

\begin{theorem}\label{degiorgi}
Set $u\in\cD\cG_{\sI}^+(Q)$. For any $p>0$, there exist constants $\g,C>0$ only depending on $p,\rm{Vol} (Q_2)$ and $\bbp_i,\l,C_\bbp$ in the definition of $\cD\cG_{\sI}^+(Q)$ such that for all $1\le\tau<\s\le2$,
\begin{align}
\|u^+\bbone_{Q_\tau}\|_{\mL^\infty}\lesssim_C (\s-\tau)^{-\g}\|u^+\bbone_{Q_\s}\|_{\mL^p}+\cA.
\end{align}
\end{theorem}

\section{Local boundedness of weak (sub)-solutions of kinetic equations}\label{sec:local-boundedness}

In this section, we establish local boundedness and energy estimates for weak sub-solutions of the kinetic equation. The proof combines the divergence-free energy cancellation and the abstract De Giorgi iteration stated in Section~\ref{sec:preliminary}.

Consider the following kinetic equation:
\begin{align}\label{kineticpde}
\p_t u=\Delta_v u+v\cdot\nn_x u+b\cdot\nn_v u+f.
\end{align}
Here $b:\mR^{1+2d}\to\mR^d$ satisfies $\div_v b=0$, $f:\mR^{1+2d}\to\mR$, and for every bounded set $Q\subset\mR^{1+2d}$,
\begin{align}\label{local-integrability}
b\bbone_Q\in\mL^2,\quad f\bbone_Q\in\mL^1,
\end{align}
where, for $p,q\in[1,\infty]$,
$$\mL_t^q(\mL_z^p):=L^q(\mR;L^p(\mR^{2d})),\quad\mL^p:=\mL_t^p(\mL_z^p).$$
We also introduce the following space of solutions: For every open set $Q\subset\mR^{1+2d}$
$$\sV_Q:=\{f\in\mL_{\text{loc}}^1:\|f\|_{\sV_Q}:=\|\bbone_{Q}f\|_{\mL_t^\infty \mL_z^2}+\|\bbone_Q \nn_v f\|_{\mL^2}<\infty\},\quad\sV_{\rm loc}:=\cap_{Q \text{bounded}}\sV_Q.$$
For given $t\in\mR$ and $r>0$,
$$\cI_t:=\bbone_{(-\infty,t]},\quad Q_r:=\{(t,x,v):|t|<r^2,|x|<r^3,|v|<r\}.$$

Now we introduce the definition of weak (sub)-solutions of kinetic equation \eqref{kineticpde}. 
\begin{definition}\label{weaksubsol}
Let $Q\subset\mR^{1+2d}$ be a bounded domain. A function $u\in\sV_Q\cap\mL_Q^\infty$ is called a weak sub-solution of PDE \eqref{kineticpde} in $Q$ if for any nonnegative $\varphi\in C_c^\infty(Q)$,
\begin{align}\label{weaksubsolineq}
-\int_Q u\p_t\varphi\le\int_Q u\Delta_v \varphi-\int_Q u(v\cdot\nn_x\varphi)+\int_Q(b\cdot\nn_v u)\varphi+\int_Q f\varphi.
\end{align}
We say that $u$ is a weak solution to \eqref{kineticpde} in $Q$ if \eqref{weaksubsolineq} holds with equality for every $\varphi\in C_c^\infty(Q)$. If $u\in\mL_{\rm loc}^\infty\cap\sV_{\rm loc}$ is a weak solution in every bounded domain $Q$, then we call $u$ a global weak solution.
\end{definition}

The following lemma is from \cite[Lemma 3.3 and Lemma 3.5]{Zhang25}.

\bl\label{weaksubsollemma}
Let $u\in\sV_Q\cap\mL_Q^\infty$ be a weak sub-solution of PDE \eqref{kineticpde} in $Q$. Then
\begin{enumerate}[]
\item \rm{(i)} If $u$ is nonnegative, then for any nonnegative $\eta\in C_c^\infty(Q)$ and $t\in\mR$,
\begin{align}
\begin{split}
&\ff12\int_{\mR^{2d}}|(u\eta)(t)|^2+\int_{\mR^{1+2d}}|\nn_v(u\eta)|^2\cI_t\\
&\qquad\le\int_{\mR^{1+2d}}u^2\big[(\eta\big(\p_s \eta-v\cdot\nn_x \eta-b\cdot\nn_v\eta\big)+|\nn_v\eta|^2)\big]\cI_t+\int_{\mR^{1+2d}}fu\eta^2\cI_t.
\end{split}
\end{align}
\item \rm{(ii)} For any $\k\ge0$, $(u-\k)^+$ is still a weak sub-solution of PDE \eqref{kineticpde} with $f\bbone_{\{u>\k\}}$ in place of $f$.
\end{enumerate}
\el

\subsection{Gain of regularity via the Duhamel formula}

Consider the drift-free Cauchy problem:
\begin{align}\label{nodriftpde}
\p_t u=\Delta_v u+v\cdot\nn_x u+f,\quad u|_{t\le0}=0. \end{align}
Let $W_t$ be a $d$-dimensional standard Brownian motion. Define the kinetic semigroup $$P_t f(x,v):=\mE f(x+tv+\sqrt{2}\int_0^tW_s\dif s,v+\sqrt2 W_t).$$
Then by the Duhamel formula, the unique solution to \eqref{nodriftpde} can be represented by
\begin{align}\label{duhamel}
u(t,z)=\int_0^t P_{t-s} f(s,z)\dif s.
\end{align}

The following kinetic Schauder-type estimate and anisotropic interpolation inequality are taken from \cite[Lemmas 3.6 and 3.7]{Zhang25}, respectively.

\bl\label{semies}
Set $1\le q\le \nu\le\infty$, $\g\in\mR$ and $\bbp\in[1,\infty]^2$. For any $T>0$ and $\beta<\g+2(1+\ff1\nu-\ff1q)$, there exists a constant $C:=C(T,d,\bbp,q,\nu,\g,\beta)>0$ such that for $t\in(0,T]$,
\begin{align}
\|u\cI_t\|_{\mL_t^\nu \bB_{\bbp,\bba}^\beta}\lesssim_C \|f\cI_t\|_{\mL_t^q \bB_{\bbp,\bba}^\g},
\end{align}
where $u$ is defined in \eqref{duhamel}.
\el

\bl\label{interpolation}
For any $\b>0$, $1\le q\le \nu\le\infty$ and $\bbr\in[2,\infty]^2$ with
$$\ff2\bbr\ge\big(1-\ff q\nu\big)\bbone,\quad \bba\cdot\big(\ff d\bbr-\ff d{\bb2}\big)+\ff{q\b}\nu>0,$$
there exists a constant $C:=C(\b,\nu,q,\bbr)>0$ such that
\begin{align}
\|u\|_{\mL_t^\nu\mL_z^\bbr}\lesssim_C \|u\|_{\mL_t^\infty\mL_z^2}+\|u\|_{\mL_t^q\bB_{\bb2,\bba}^\b}.
\end{align}
\el

\subsection{Local boundedness with $L^p$-integrable inhomogeneous $f$}

We now use the estimates from the preceding section to establish local boundedness for weak sub-solutions. Let
\begin{align}\label{qpdef}
q_1\in(2,\infty),\quad \bbp_1=(p_{1,x},p_{1,v})\in(2,\infty)^2,\quad q_0\in(1,\infty),\quad \bbp_0=(p_{0,x},p_{0,v})\in(1,\infty)^2.
\end{align}
For $i=0,1$, set
\begin{align}\label{STdef}
S_i:=\bba\cdot\frac d{\bbp_i}=\frac{3d}{p_{i,x}}+\frac d{p_{i,v}},\quad T_i:=\frac2{q_i}.
\end{align}

Throughout this section, we impose the following assumption on the the drift term $b$ and the inhomogeneous term $f$.

\noindent{\bf ($\wt{\text{H}}$)}
Assume that
$$b\in L^{q_1}\big((-4,4);L^{\bbp_1}(\mR^{2d})\big),\quad\div_v b=0,\quad\bbone_{Q_2}f\in\mL_t^{q_0}\mL_z^{\bbp_0},$$
and that there exists $\b\in(0,1)$ such that
$$\max\Big\{\frac{S_1}{2-T_1},\frac{S_0}{2-T_0}\Big\}<\b<2-S_1-T_1.$$

\begin{theorem}\label{thm:df-thm43}
Under Assumption {\bf ($\wt{\text{H}}$)}, let $u$ be a weak subsolution of \eqref{kineticpde} in $Q_2$. Then for all $|t|\leq1$,
\begin{align}\label{df43-estimate}
\|\bbone_{Q_1}u^+\cI_t\|_{\mL^\infty}+\|\bbone_{Q_1}\nn_vu^+\cI_t\|_{\mL^2}\lesssim_C\|\bbone_{Q_2}u^+\cI_t\|_{\mL^2}+\|\bbone_{Q_2}f\cI_t\|_{\mL_t^{q_0}\mL_z^{\bbp_0}}.
\end{align}
Here $C>0$ depends only on $d,q_1,\bbp_1,q_0,\bbp_0,\b$ and $\|b\|_{\mL_t^{q_1}\mL_z^{\bbp_1}((-4,4)\times\mR^{2d})}$.
\end{theorem}

\begin{proof}
Fix $t\in[-1,1]$ and set $\wt Q_\tau^t:=Q_\tau\cap\big((-\infty,t)\times\mR^{2d}\big)$. By Theorem \ref{degiorgi}, it suffices to prove that $u\in\cD\cG_{\sI_\b}^+(\wt Q_\cdot^t)$ with $\cA:=\|\bbone_{Q_2}f\cI_t\|_{\mL_t^{q_0}\mL_z^{\bbp_0}}$,
where
$$\sI_\b:=\Big\{(\nu,\bbr)\in(2,\infty)\times(2,\infty)^2:\ff2{\bbr}>\big(1-\ff2\nu\big)\bbone, \bba\cdot\big(\ff d{\bbr}-\ff d{\bb2}\big)+\ff{2\b}{\nu}>0\Big\}.$$
More precisely, we prove that for any $(\nu,\bbr)\in\sI_\b$, any $1\leq\tau<\s\leq2$ and any $\k\geq0$,
\begin{align}\label{df-DG-new}
(\s-\tau)^2\|\bbone_{Q_\tau}(u-\k)^+\cI_t\|_{\mL_t^\nu\mL_z^\bbr}\lesssim_C \sum_{i=0,1}\|\bbone_{Q_\s}(u-\k)^+\cI_t\|_{\mL_t^{\nu_i}\mL_z^{\bbr_i}} +\cA\|\bbone_{\{u>\k\}\cap Q_\s}\cI_t\|_{\mL_t^{\nu_0}\mL_z^{\bbr_0}},
\end{align}
where $\ff1{q_i}+\ff2{\nu_i}=1$ and $\ff1{\bbp_i}+\ff2{\bbr_i}=\bbone$. By assumption {\bf ($\wt{\text{H}}$)}, a straightforward calculation shows that
$$(\nu_i,\bbr_i)\in\sI_\b,\quad i=0,1.$$

{\bf{Step 1.}} Let $1\leq\tau<\s\leq2$ and $\bar\tau:=\ff{\tau+\s}2$. Choose $\eta_0\in C_c^\infty(Q_{\bar\tau};[0,1])$ such that $\eta_0=1$ on $Q_\tau$ and
\begin{align}\label{etacutoff}
\|\p_s\eta_0\|_\infty^{1/2}+\|v\cdot\nn_x\eta_0\|_\infty^{1/2}+\|\nn_v\eta_0\|_\infty\leq C(\s-\tau)^{-1},\quad \|\Delta_v\eta_0\|_\infty\leq C(\s-\tau)^{-2}.
\end{align}
By Lemma \ref{weaksubsollemma} (ii), $(u-\k)^+$ is a nonnegative weak sub-solution of \eqref{kineticpde} with $f\bbone_{\{u>\k\}}$ in place of $f$. Hence, it suffices to prove \eqref{df-DG-new} for $\k=0$ and $u\geq0$.

A direct calculation gives
$$\p_s(u\eta_0)-\Delta_v(u\eta_0)-v\cdot\nn_x(u\eta_0)-b\cdot\nn_v(u\eta_0)\leq \div_vF+f_1+f_2+f_3,$$
where
$$F:=-2u\nn_v\eta_0,\quad f_1:=u\big((\p_s-v\cdot\nn_x)\eta_0+\Delta_v\eta_0\big),\quad f_2:=f\bbone_{\{u>0\}}\eta_0,\quad f_3:=-u\,b\cdot\nn_v\eta_0.$$
By \eqref{etacutoff},
\begin{align}\label{df-Ff1-new}
\|F\cI_t\|_{\mL^2}+\|f_1\cI_t\|_{\mL^2}\lesssim_C(\s-\tau)^{-2}\|\bbone_{Q_\s}u\cI_t\|_{\mL_t^{\nu_1}\mL_z^{\bbr_1}}.
\end{align}
For $i=0,1$, define $(\bar q_i,\bar{\bbp}_i)$ by $$\ff1{\bar q_i}:=\ff1{q_i}+\ff1{\nu_i},\quad \ff1{\bar{\bbp}_i}:=\ff1{\bbp_i}+\ff1{\bbr_i}.$$ Then by the H\"older's inequality,
\begin{align}\label{df-f2-new}
\|f_2\cI_t\|_{\mL_t^{\bar q_0}\mL_z^{\bar{\bbp}_0}}\leq\|\bbone_{Q_2}f\cI_t\|_{\mL_t^{q_0}\mL_z^{\bbp_0}}\|\bbone_{\{u>0\}\cap Q_\s}\cI_t\|_{\mL_t^{\nu_0}\mL_z^{\bbr_0}}.
\end{align}
Similarly,
\begin{align}\label{df-f3-new}
\begin{split}
\|f_3\cI_t\|_{\mL_t^{\bar q_1}\mL_z^{\bar{\bbp}_1}}&\lesssim_C(\s-\tau)^{-1}\|\bbone_{Q_2}b\|_{\mL_t^{q_1}\mL_z^{\bbp_1}}\|\bbone_{Q_\s}u\cI_t\|_{\mL_t^{\nu_1}\mL_z^{\bbr_1}}\\
&\lesssim_C (\s-\tau)^{-1}\|b\|_{\mL_t^{q_1}\mL_z^{\bbp_1}((-4,4)\times\mR^{2d})}\|\bbone_{Q_\s}u\cI_t\|_{\mL_t^{\nu_1}\mL_z^{\bbr_1}}.
\end{split}
\end{align}

{\bf{Step 2.}} We extend $b$ and $f$ by zero in time outside $(-4,4)$ and continue to denote the extension by $b$. By a standard continuity method, there exists a  global solution $w$ to
\begin{align}\label{df-cauchy-eq-new}
\p_sw-\Delta_vw-v\cdot\nn_xw-b\cdot\nn_vw=\div_vF+f_1+f_2+f_3,\quad w|_{s\leq-\bar\tau^2}=0.
\end{align}
Denote $$A_t^\b:=\|w\cI_t\|_{\mL_t^\infty\mL_z^2}+\|\nn_vw\cI_t\|_{\mL^2}+\|w\cI_t\|_{\mL_t^2\bB_{\bb2,\bba}^{\b}}.$$
We claim that
\begin{align}\label{df-cauchy-new}
A_t^\b\lesssim_C\|F\cI_t\|_{\mL^2}+\|f_1\cI_t\|_{\mL^2}+\|f_2\cI_t\|_{\mL_t^{\bar q_0}\mL_z^{\bar{\bbp}_0}}+\|f_3\cI_t\|_{\mL_t^{\bar q_1}\mL_z^{\bar{\bbp}_1}}.
\end{align}
By a standard approximation argument, it suffices to prove \eqref{df-cauchy-new} for smooth $b$, $F$, and $f_i$. Let $b^n$, $F^n$, and $f_i^n$ denote their standard space--time mollifications. Then $\div_v b^n=0$, and the relevant norm bounds hold uniformly in $n$. The estimate for the general case follows by passing to the limit as $n\to\infty$. Hence, in what follows, we assume that $b$, $F$, and $f_i$ are smooth.

Let $\chi\in C_c^\infty(\mR^{1+2d};[0,1])$ satisfy $\chi=1$ on $Q_1$ and $\chi=0$ on $Q_2^c$. For $R\geq1$, set
$$\chi_R(s,x,v):=\chi\left(\tfrac{s}{R^2},\tfrac{x}{R^3},\tfrac{v}{R}\right).$$
Then,
\begin{align}\label{bound}
\|\nn_v\chi_R\|_\infty\lesssim \tfrac1R,\quad\|\p_s\chi_R-v\cdot\nn_x\chi_R\|_\infty\lesssim \tfrac1{R^2}.
\end{align}

Fix $\theta\in(-\bar\tau^2,t]$. We use $w\chi_R^2$ as a test function in \eqref{df-cauchy-eq-new}. For the drift term, the divergence-free condition gives
$$\int_{\mR^{2d}}(b\cdot\nn_vw)w\chi_R^2\,\dif z=-\int_{\mR^{2d}}w^2\chi_R\,b\cdot\nn_v\chi_R\,\dif z.$$
Applying the Young's inequality, we obtain
\begin{align*}
&\|w(\theta)\chi_R(\theta)\|_{\mL_z^2}^2+\int_{-\bar\tau^2}^{\theta}\int_{\mR^{2d}}|\nn_vw|^2\chi_R^2\,\dif z\dif s\\
&\quad\lesssim\int_{-\bar\tau^2}^{\theta}\int_{\mR^{2d}}|F\cdot\nn_vw|\chi_R^2\,\dif z\dif s+\sum_{i=1}^3\int_{-\bar\tau^2}^{\theta}\int_{\mR^{2d}}|f_iw|\chi_R^2\,\dif z\dif s\\
&\qquad+\ff1{R^2}\|w\cI_t\|_{\mL^2}^2+\ff1R\|F\cI_t\|_{\mL^2}\|w\cI_t\|_{\mL^2}+\ff1R\|b\|_{\mL^\infty}\|w\cI_t\|_{\mL^2}^2,
\end{align*}
Letting $R\to\infty$, we obtain
\begin{align}\label{df-energy-w-new}
\|w\cI_t\|_{\mL_t^\infty\mL_z^2}^2+\|\nn_vw\cI_t\|_{\mL^2}^2\lesssim\int_{\mR^{1+2d}}|F\cdot\nn_vw\cI_t|\,\dif z\dif s+\sum_{i=1}^3\int_{\mR^{1+2d}}|f_iw\cI_t|\,\dif z\dif s.
\end{align}
By the H\"older's inequality and the Young's inequality,
$$\int |F\cdot\nn_vw\,\cI_t|\leq\|F\cI_t\|_{\mL^2}\|\nn_vw\cI_t\|_{\mL^2}\leq\frac14\|\nn_vw\cI_t\|_{\mL^2}^2+C\|F\cI_t\|_{\mL^2}^2.$$
Similarly, for $i=1$,
$$\int |f_1w\,\cI_t|\leq\|f_1\cI_t\|_{\mL^2}\|w\cI_t\|_{\mL^2}\lesssim\|f_1\cI_t\|_{\mL^2}\|w\cI_t\|_{\mL_t^\infty\mL_z^2}\leq\frac14\|w\cI_t\|_{\mL_t^\infty\mL_z^2}^2+C\|f_1\cI_t\|_{\mL^2}^2,$$
and for $i=2,3$, 
$$\int |f_2 w\,\cI_t|\leq\|f_2\cI_t\|_{\mL_t^{\bar q_0}\mL_z^{\bar{\bbp}_0}}\|w\cI_t\|_{\mL_t^{\nu_0}\mL_z^{\bbr_0}},\quad \int |f_3 w\,\cI_t|\leq\|f_3\cI_t\|_{\mL_t^{\bar q_1}\mL_z^{\bar{\bbp}_1}}\|w\cI_t\|_{\mL_t^{\nu_1}\mL_z^{\bbr_1}}.$$
Applying the Young's inequality and using the estimates above, we obtain the following energy estimate: for every $\eps\in(0,1)$,
\begin{align}\label{eq-energy}
\begin{split}
&\|w\cI_t\|_{\mL_t^\infty\mL_z^2}+\|\nn_vw\cI_t\|_{\mL^2}\\
&\qquad\lesssim_C\|F\cI_t\|_{\mL^2}+\|f_1\cI_t\|_{\mL^2}+\eps\sum_{i=0,1}\|w\cI_t\|_{\mL_t^{\nu_i}\mL_z^{\bbr_i}}+C_\eps(\|f_2\cI_t\|_{\mL_t^{\bar q_0}\mL_z^{\bar{\bbp}_0}}+\|f_3\cI_t\|_{\mL_t^{\bar q_1}\mL_z^{\bar{\bbp}_1}}).
\end{split}
\end{align}

We now estimate $\|w\cI_t\|_{\mL_t^2\bB_{\bb2,\bba}^{\b}}$. The  Duhamel's formula gives
$$w(s)=\int_{-\bar\tau^2}^sP_{s-r}\big(b\cdot\nn_vw+\div_vF+f_1+f_2+f_3\big)(r)\dif r=:J_b+J_F+J_1+J_2+J_3.$$
We apply Lemma \ref{semies} to the terms above with the parameter choices specified below.

For $J_F$, take $(\g,q,\nu,\bbp)=(-1,2,2,\bb2)$. The Bernstein inequality then gives
$$\|J_F\cI_t\|_{\mL_t^2\bB_{\bb2,\bba}^{\b}}\lesssim\|\div_vF\cI_t\|_{\mL_t^2\bB_{\bb2,\bba}^{-1}}\lesssim\|F\cI_t\|_{\mL^2}.$$
For $J_1$, taking $(\g,q,\nu,\bbp)=(0,2,2,\bb2)$ yields
$$\|J_1\cI_t\|_{\mL_t^2\bB_{\bb2,\bba}^{\b}}\lesssim\|f_1\cI_t\|_{\mL^2}.$$
For $J_2$ and $J_3$, set $\alpha_i:=\b+\bba\cdot(\frac d{\bar{\bbp}_i}-\frac d{\bb2})$, $i=0,1$. By the Bernstein embedding,
$$\|J_2\cI_t\|_{\mL_t^2\bB_{\bb2,\bba}^{\b}}\lesssim\|J_2\cI_t\|_{\mL_t^2\bB_{\bar{\bbp}_0,\bba}^{\alpha_0}},\quad \|J_3\cI_t\|_{\mL_t^2\bB_{\bb2,\bba}^{\b}}\lesssim\|J_3\cI_t\|_{\mL_t^2\bB_{\bar{\bbp}_1,\bba}^{\alpha_1}}.$$
Note that $\alpha_i=\b+\bba\cdot(\frac d{\bar{\bbp}_i}-\frac d{\bb2})<3-\frac2{\bar q_i}$. Thus taking $(\g,q,\nu,\bbp)=(0,\bar q_i,2,\bar{\bbp}_i)$ yields
$$\|J_2\cI_t\|_{\mL_t^2\bB_{\bb2,\bba}^{\b}}\lesssim_C\|f_2\cI_t\|_{\mL_t^{\bar q_0}\mL_z^{\bar{\bbp}_0}},\quad \|J_3\cI_t\|_{\mL_t^2\bB_{\bb2,\bba}^{\b}}\lesssim_C\|f_3\cI_t\|_{\mL_t^{\bar q_1}\mL_z^{\bar{\bbp}_1}}.$$

It remains to estimate $J_b$. Define $(\hat q,\hat{\bbp})$ by $\frac1{\hat q}:=\frac1{q_1}+\frac12$, $\frac1{\hat{\bbp}}:=\frac1{\bbp_1}+\frac1{\bb2}$. Then $\hat{\bbp}\in(1,2)^2$ and $\hat q\in(1,2)$. Set $\alpha_b:=\b+\bba\cdot(\frac d{\hat{\bbp}}-\frac d{\bb2})=\b+S_1<2-T_1=3-\frac2{\hat q}$. Hence, taking $(\g,q,\nu,\bbp)=(0,\hat q,2,\hat{\bbp})$,
\begin{align*}
\|J_b\cI_t\|_{\mL_t^2\bB_{\bb2,\bba}^{\b}}\lesssim\|J_b\cI_t\|_{\mL_t^2\bB_{\hat{\bbp},\bba}^{\alpha_b}}\lesssim\|b\cdot\nn_vw\,\cI_t\|_{\mL_t^{\hat q}\mL_z^{\hat{\bbp}}}.
\end{align*}
By the H\"older's inequality and {\bf ($\wt{\text{H}}$)},
\begin{align}\label{df-Jb-new}
\|b\cdot\nn_vw\,\cI_t\|_{\mL_t^{\hat q}\mL_z^{\hat{\bbp}}}\leq\|b\cI_t\|_{\mL_t^{q_1}\mL_z^{\bbp_1}((-4,4)\times\mR^{2d})}\|\nn_vw\cI_t\|_{\mL^2}.
\end{align}
Combining the above estimates with \eqref{eq-energy} yields, for any $\eps\in(0,1)$,
$$A_t^\b\lesssim_C\|F\cI_t\|_{\mL^2}+\|f_1\cI_t\|_{\mL^2}+\eps\sum_{i=0,1}\|w\cI_t\|_{\mL_t^{\nu_i}\mL_z^{\bbr_i}}+C_\eps(\|f_2\cI_t\|_{\mL_t^{\bar q_0}\mL_z^{\bar{\bbp}_0}}+\|f_3\cI_t\|_{\mL_t^{\bar q_1}\mL_z^{\bar{\bbp}_1}}).$$
Since $(\nu_i,\bbr_i)\in\sI_\b$, Lemma \ref{interpolation} gives
$$\|w\cI_t\|_{\mL_t^{\nu_i}\mL_z^{\bbr_i}}\lesssim_C\|w\cI_t\|_{\mL_t^\infty\mL_z^2}+\|w\cI_t\|_{\mL_t^2\bB_{\bb2,\bba}^{\b}}\lesssim_C A_t^\b,\quad i=0,1.$$
Taking $\eps>0$ sufficiently small, we obtain \eqref{df-cauchy-new}.

Combining \eqref{df-cauchy-new} with \eqref{df-Ff1-new}, \eqref{df-f2-new} and \eqref{df-f3-new}, we get
\begin{align}\label{df-At-new}
(\s-\tau)^2 A_t^\b\lesssim_C\sum_{i=0,1}\|\bbone_{Q_\s}u\cI_t\|_{\mL_t^{\nu_i}\mL_z^{\bbr_i}}+\|\bbone_{Q_2}f\cI_t\|_{\mL_t^{q_0}\mL_z^{\bbp_0}}\|\bbone_{\{u>0\}\cap Q_\s}\cI_t\|_{\mL_t^{\nu_0}\mL_z^{\bbr_0}}.
\end{align}

{\bf{Step 3.}} Set $\bar w:=u\eta_0-w$. Then in the sense of distributions,
\begin{align}\label{Kbbarw}
\p_s\bar w-\Delta_v\bar w-v\cdot\nn_x\bar w-b\cdot\nn_v\bar w\leq0,\quad\bar w|_{s\leq-\bar\tau^2}=0.
\end{align}
Since $u\geq0$, we have $0\leq\bar w^+\leq u\eta_0+|w|$. Hence, by Lemma \ref{interpolation} and \eqref{df-cauchy-new},
$$\|\bar w^+\cI_t\|_{\mL_t^{\nu_1}\mL_z^{\bbr_1}}\leq\|u\eta_0\cI_t\|_{\mL_t^{\nu_1}\mL_z^{\bbr_1}}+\|w\cI_t\|_{\mL_t^{\nu_1}\mL_z^{\bbr_1}}\lesssim_C\|u\eta_0\cI_t\|_{\mL_t^{\nu_1}\mL_z^{\bbr_1}}+A_t^\b<\infty.$$

For $M>0$, set
$$T_M(r):=r^+\wedge M,\quad \Phi_M(r):=\int_0^rT_M(\rho)\dif\rho.$$
By a standard regularization argument, testing \eqref{Kbbarw} against $T_M(\bar w)\chi_R^2\mathcal I_\theta$ yields
\begin{align*}
&\int_{\mathbb R^{2d}}\Phi_M(\bar w(\theta))\chi_R^2(\theta)
+\int T_M'(\bar w)|\nabla_v\bar w|^2\chi_R^2\mathcal I_\theta
\\
&\qquad\lesssim
\int \Phi_M(\bar w)
\left(
|(\partial_s-v\cdot\nabla_x)\chi_R^2|
+|\Delta_v\chi_R^2|
\right)\mathcal I_\theta
+\int \Phi_M(\bar w)
|b\cdot\nabla_v(\chi_R^2)|\mathcal I_\theta.
\end{align*}
Using the fact that $0\leq\Phi_M(r)\leq (r^+)^2/2$, together with \eqref{bound} and H\"older's inequality, we obtain, for almost every $\theta\leq t$, 
\begin{align*}
\int_{\mR^{2d}}\Phi_M(\bar w(\theta))\chi_R^2+\int T_M'(\bar w)|\nn_v\bar w|^2\chi_R^2\cI_\theta\lesssim \ff1{R^2}\|\bar w^+\cI_t\|_{\mL^2}^2+\ff1R\|b\|_{\mL_s^{q_1}\mL_z^{\bbp_1}((-4,4)\times\mR^{2d})}\|\bar w^+\cI_t\|_{\mL_s^{\nu_1}\mL_z^{\bbr_1}}^2.
\end{align*}
Letting first $R\to\infty$ and then $M\to\infty$, we get
$$\frac12\|\bar w^+(\theta)\|_{\mL_z^2}^2+\int_{-\infty}^{\theta}\|\nn_v\bar w^+(s)\|_{\mL_z^2}^2\dif s\leq0.$$
Hence $\bar w^+\cI_t=0$, and therefore $u\eta_0\leq w$.

For any $(\nu,\bbr)\in\sI_\b$, Lemma \ref{interpolation} yields
$$\|\bbone_{Q_\tau}u\cI_t\|_{\mL_t^\nu\mL_z^\bbr}\leq\|w\cI_t\|_{\mL_t^\nu\mL_z^\bbr}\lesssim_C A_t^\b.$$
Together with \eqref{df-At-new}, this proves \eqref{df-DG-new}. Hence, $u\in\cD\cG_{\sI_\b}^+(\wt Q_\cdot^t)$. By Theorem \ref{degiorgi}, we obtain
\begin{align}\label{df-linfty-new}
\|\bbone_{Q_{3/2}}u^+\cI_t\|_{\mL^\infty}\lesssim_C\|\bbone_{Q_2}u^+\cI_t\|_{\mL^2}+\|\bbone_{Q_2}f\cI_t\|_{\mL_t^{q_0}\mL_z^{\bbp_0}}.
\end{align}
It remains to estimate $\nn_vu^+$. Let $\chi\in C_c^\infty(Q_{3/2};[0,1])$ satisfy $\chi=1$ on $Q_1$. Applying Lemma \ref{weaksubsollemma} (i) to $u^+$ and then using the H\"older's inequality, we obtain
$$\|\nn_v(u^+\chi)\cI_t\|_{\mL^2}\lesssim_C\|\bbone_{Q_{3/2}}u^+\cI_t\|_{\mL_t^{\nu_1}\mL_z^{\bbr_1}}+\|\<\!\<f\chi,u^+\chi\>\!\>\cI_t\|_{\mL_t^1}^{1/2}.$$
By \eqref{df-linfty-new},
$$\|\bbone_{Q_{3/2}}u^+\cI_t\|_{\mL_t^{\nu_1}\mL_z^{\bbr_1}}\lesssim_C\|\bbone_{Q_{3/2}}u^+\cI_t\|_{\mL^\infty}\lesssim_C\|\bbone_{Q_2}u^+\cI_t\|_{\mL^2}+\|\bbone_{Q_2}f\cI_t\|_{\mL_t^{q_0}\mL_z^{\bbp_0}}.$$
Moreover,
$$\|\<\!\<f\chi,u^+\chi\>\!\>\cI_t\|_{\mL_t^1}\leq\|\bbone_{Q_2}f\cI_t\|_{\mL^1}\|\bbone_{Q_{3/2}}u^+\cI_t\|_{\mL^\infty}\lesssim\|\bbone_{Q_2}f\cI_t\|_{\mL_t^{q_0}\mL_z^{\bbp_0}}\|\bbone_{Q_{3/2}}u^+\cI_t\|_{\mL^\infty}.$$
By \eqref{df-linfty-new} and the Young's inequality, we obtain
$$\|\<\!\<f\chi,u^+\chi\>\!\>\cI_t\|_{\mL_t^1}^{1/2}\lesssim_C\|\bbone_{Q_2}u^+\cI_t\|_{\mL^2}+\|\bbone_{Q_2}f\cI_t\|_{\mL_t^{q_0}\mL_z^{\bbp_0}}.$$
Since $\chi=1$ on $Q_1$,
$$\|\bbone_{Q_1}\nn_vu^+\cI_t\|_{\mL^2}\lesssim_C\|\bbone_{Q_2}u^+\cI_t\|_{\mL^2}+\|\bbone_{Q_2}f\cI_t\|_{\mL_t^{q_0}\mL_z^{\bbp_0}}.$$
Combining this with \eqref{df-linfty-new} proves \eqref{df43-estimate}.
\end{proof}

\section{Global boundedness and well-posedness of weak solutions}\label{sec:global-wellposedness}

In this section, we use the local estimate established in Theorem \ref{thm:df-thm43} to prove global boundedness and well-posedness for the Cauchy problem. We first introduce the notation needed to apply this estimate uniformly in space and time.

Fix $T>0$. For $t_0\in[0,T]$, $z_0=(x_0,v_0)\in\mR^{2d}$ and a function $g$ on $\mR^{1+2d}$, we define the Galilean translation compatible with \eqref{kineticpde} by
\begin{align}\label{df-galilean-shift}
g^{t_0,z_0}(s,x,v):=g(t_0+s,x+x_0-sv_0,v+v_0)
\end{align}
and the localized kinetic energy space by
$$\wt\sV_T:=\big\{g\in\mL_{\rm loc}^1:\nor g\nor_{\wt\sV_T}:=\sup_{t_0\in[0,T],\,z_0\in\mR^{2d}}\big(\|\bbone_{Q_1}g^{t_0,z_0}\|_{\mL_s^\infty\mL_z^2}+\|\bbone_{Q_1}\nn_vg^{t_0,z_0}\|_{\mL^2}\big)<\infty\big\}.$$
Replacing $Q_1$ by any fixed $Q_r$ yields an equivalent norm.

For $t_0\in[0,T]$, the time projection of $Q_2(t_0,z_0)$ is contained in $(-4,T+4)$. Hence, the drift norm required in the local estimate is bounded by $\|b\cI_{T+4}\|_{\mL_t^{q_1}\mL_z^{\bbp_1}}$, yielding the following result.

\begin{theorem}\label{thm:df-global}
Under Assumption {\bf (${\text{H}}$)}, the Cauchy problem \eqref{kineticpde} with $u|_{t\leq0}=0$ admits a unique global weak solution $u$ in the sense of Definition \ref{weaksubsol}, satisfying
\begin{align}\label{unicondition}
\|u\cI_T\|_{\mL^\infty}+\nor u\cI_T\nor_{\wt\sV_T}<\infty,\quad T>0.
\end{align}
Moreover, for every $T>0$, there exists a constant $C_T=C(T,d,q_1,\bbp_1,q_0,\bbp_0,\b,\|b\cI_{T+4}\|_{\mL_t^{q_1}\mL_z^{\bbp_1}})>0$ such that
\begin{align}\label{df-global-estimate}
\|u\cI_t\|_{\mL^\infty}+\nor u\cI_t\nor_{\wt\sV_T}\lesssim_{C_T}\|f\cI_t\|_{\mL_t^{q_0}\mL_z^{\bbp_0}},\quad t\in[0,T].
\end{align}
\end{theorem}

\begin{proof}
We divide the proof into two steps.

{\bf Step 1.} We first prove the a priori estimate and uniqueness. Fix $T>0$ and let $u$ be a global weak solution satisfying \eqref{unicondition}. For $t_0\in[0,T]$ and $z_0\in\mR^{2d}$, define $u^{t_0,z_0}$, $b^{t_0,z_0}$ and $f^{t_0,z_0}$ by \eqref{df-galilean-shift}. A direct calculation shows that
\begin{align*}
\p_t u^{t_0,z_0}=\Delta_vu^{t_0,z_0}+v\cdot\nn_xu^{t_0,z_0}+b^{t_0,z_0}\cdot\nn_vu^{t_0,z_0}+f^{t_0,z_0}.
\end{align*}
Moreover, $\div_vb^{t_0,z_0}=0$ and $\|b^{t_0,z_0}\|_{\mL_s^{q_1}\mL_z^{\bbp_1}((-4,4)\times\mR^{2d})}\leq\|b\cI_{T+4}\|_{\mL_t^{q_1}\mL_z^{\bbp_1}}$. Applying Theorem \ref{thm:df-thm43} to $\pm u^{t_0,z_0}$, we obtain
\begin{align}\label{df-global-local-estimate}
\|\bbone_{Q_1}u^{t_0,z_0}\cI_r\|_{\mL^\infty}+\|\bbone_{Q_1}\nn_vu^{t_0,z_0}\cI_r\|_{\mL^2}\lesssim_{C_T}\|\bbone_{Q_2}u^{t_0,z_0}\cI_r\|_{\mL^2}+\|f\cI_{t_0+r}\|_{\mL_t^{q_0}\mL_z^{\bbp_0}}.
\end{align}

Taking $t_0=t$ and $r=0$ in \eqref{df-global-local-estimate}, and covering $Q_2$ by finitely many translates of $B_1^\bba$, we obtain
\begin{align}\label{df-global-gronwall}
\sup_{z_0\in\mR^{2d}}\|\bbone_{B_1^\bba(z_0)}u(t)\|_{\mL_z^2}^2\lesssim_{C_T}\int_0^t\sup_{z_0\in\mR^{2d}}
\|\bbone_{B_1^\bba(z_0)}u(s)\|_{\mL_z^2}^2\dif s+\|f\cI_t\|_{\mL_t^{q_0}\mL_z^{\bbp_0}}^2,\quad t\in[0,T].
\end{align}
Indeed, for each fixed $s$, the map $z\mapsto(x+x_0-sv_0,v+v_0)$ is a translation, and the number of balls in the above covering is independent of $s$ and $z_0$. Hence, the Gronwall inequality yields
\begin{align}\label{df-global-uloc-l2}
\sup_{s\in[0,t]}\sup_{z_0\in\mR^{2d}}\|\bbone_{B_1^\bba(z_0)}u(s)\|_{\mL_z^2}\lesssim_{C_T}\|f\cI_t\|_{\mL_t^{q_0}\mL_z^{\bbp_0}},\quad t\in[0,T].
\end{align}
Substituting this estimate into \eqref{df-global-local-estimate}, taking the supremum over $t_0$ and $z_0$, and using a finite covering once more, we obtain \eqref{df-global-estimate}. 

If $u_1$ and $u_2$ are two global weak solutions satisfying \eqref{unicondition}, then $w:=u_1-u_2$ satisfies the same equation with $f=0$ and $w|_{t\leq0}=0$. Applying \eqref{df-global-gronwall} to $w$ gives $\sup_{z_0\in\mR^{2d}}\|\bbone_{B_1^\bba(z_0)}w(t)\|_{\mL_z^2}^2\equiv0$, and hence $u_1=u_2$.

{\bf Step 2.} We now prove the existence. As above, let $b^n,f^n\in C_b^\infty$ be smooth approximations of $b$ and $f$, respectively. Then $b^n$ remains divergence-free, $b^n\to b$ in $\mL_{\rm loc}^2$, and $f^n\to f$ in $\mL_{\rm loc}^1$. Moreover,
$$\sup_n\big(\|b^n\|_{\mL_t^{q_1}\mL_z^{\bbp_1}((-4,T+4)\times\mR^{2d})}+\|f^n\cI_T\|_{\mL_t^{q_0}\mL_z^{\bbp_0}}\big)\lesssim\|b\cI_{T+4}\|_{\mL_t^{q_1}\mL_z^{\bbp_1}}+\|f\cI_T\|_{\mL_t^{q_0}\mL_z^{\bbp_0}}.$$
For each $n$, the smooth Cauchy problem
\begin{align}\label{df-global-approx}
\p_tu^n=\Delta_vu^n+v\cdot\nn_xu^n+b^n\cdot\nn_vu^n+f^n,\quad u^n|_{t\leq0}=0,
\end{align}
has a classical solution. By Step 1,
\begin{align*}
\sup_n\big(\|u^n\cI_T\|_{\mL^\infty}+\nor u^n\cI_T\nor_{\wt\sV_T}\big)<\infty.
\end{align*}
Thus, by a standard weak convergence argument, there are a subsequence, still denoted by $u^n$, and a function $u$ such that, for every bounded domain $Q\subset(-\infty,T)\times\mR^{2d}$,
\begin{align*}
u^n\rightharpoonup u\hbox{ in }\mL^\infty(Q),\quad\nn_vu^n\rightharpoonup\nn_vu\hbox{ in }\mL^2(Q).
\end{align*}

It suffices to pass the limit in the drift term. Note that $b^n\to b$ strongly in $\mL^2(Q)$. Hence, for every $\varphi\in C_c^\infty(Q)$,
\begin{align*}
\Big|\int_Q((b^n-b)\cdot\nn_vu^n)\varphi\Big|\leq\|\bbone_Q(b^n-b)\|_{\mL^2}\|\bbone_Q\nn_vu^n\|_{\mL^2}\|\varphi\|_\infty\to0,
\end{align*}
while the weak convergence of $\nn_vu^n$ gives
\begin{align*}
\int_Q b\cdot\nn_v(u^n-u)\varphi\to0.
\end{align*}
Passing to the limit in the weak formulation shows that $u$ is a weak solution with $u|_{t\leq0}=0$, while \eqref{df-global-estimate} follows from weak lower semicontinuity.

Since $T>0$ is arbitrary, uniqueness guarantees the consistency of these solutions, which therefore define a global weak solution on $\mR_+$.
\end{proof}

\end{document}